\documentclass[11pt]{amsart}

\newtheorem{thm}{Theorem}[section]
\newtheorem*{thm*}{Theorem}

\newtheorem{lem}[thm]{Lemma}
\newtheorem*{lem*}{Lemma}
\newtheorem{mainthm}{Theorem}
\newtheorem*{mainthm*}{Theorem}

\theoremstyle{definition}

\newtheorem*{case*}{Case}
\newtheorem{conj}[thm]{Conjecture}

\newtheorem{defn}[thm]{Definition}
\newtheorem*{defn*}{Definition}
\newtheorem{exmp}[thm]{Example}
\newtheorem*{exmp*}{Example}

\renewcommand{\thestep}{}

\theoremstyle{remark}

\renewcommand{\thecase}{}

\newtheorem{rmk}[thm]{Remark}
\newtheorem*{rmk*}{Remark}

\makeatletter
\def\alphenumi{
  \def\theenumi{\alph{enumi}}
  \def\p@enumi{\theenumi}
  \def\labelenumi{(\@alph\c@enumi)}}
\makeatother

\makeatletter
\def\thecase{\@arabic\c@case}
\makeatother

\makeatletter
\def\thestep{\@arabic\c@step}
\makeatother

\newcount\hh
\newcount\mm
\mm=\time
\hh=\time
\divide\hh by 60
\divide\mm by 60
\multiply\mm by 60
\mm=-\mm
\advance\mm by \time
\def\hhmm{\number\hh:\ifnum\mm<10{}0\fi\number\mm}

\let\oldmarginpar\marginpar
\renewcommand\marginpar[1]{\-\oldmarginpar[\raggedleft\footnotesize #1]%
{\raggedright\footnotesize #1}}

\newcommand\RR{\mathbb{R}}

\newcommand\sA{{\mathscr{A}}}

\newcommand\sD{{\mathscr{D}}}
\newcommand\sE{{\mathscr{E}}}

\newcommand\sH{{\mathscr{H}}}

\newcommand\sM{{\mathscr{M}}}

\newcommand\sU{{\mathscr{U}}}

\newcommand\sX{{\mathscr{X}}}

\newcommand\eps{\varepsilon}

\newcommand\U{\operatorname{U}}

\newcommand\less{\setminus}

\newcommand\divg{\operatorname{div}}

\newcommand{\esssup}{\operatornamewithlimits{ess\ sup}}

\newcommand\grad{\operatorname{grad}}

\newcommand\tr{\operatorname{tr}}

\newcommand\vol{\operatorname{vol}}

\newcommand\apriori{{\emph{a priori }}}
\newcommand\Apriori{{\emph{A priori }}}

\newcommand\round{{\mathrm{round}}}

\numberwithin{equation}{section}
\numberwithin{figure}{section}

\usepackage{amscd}
\usepackage{amssymb}
\usepackage{graphicx}
\usepackage{hyperref}
\usepackage{mathrsfs}
\usepackage[usenames]{color}
\usepackage{paralist}
\usepackage{slashed}
\usepackage{url}
\usepackage{verbatim}
\usepackage{xr}

\hypersetup{pdftitle={Relative energy gap for harmonic maps of Riemann surfaces into real analytic Riemannian manifolds}}
\hypersetup{pdfauthor={Paul M. N. Feehan}}

\begin{document}

\title[Relative energy gap for harmonic maps]{Relative energy gap for harmonic maps of Riemann surfaces into real analytic Riemannian manifolds}

\author[Paul M. N. Feehan]{Paul M. N. Feehan}
\address{Department of Mathematics, Rutgers, The State University of New Jersey, 110 Frelinghuysen Road, Piscataway, NJ 08854-8019, United States of America}
\email{feehan@math.rutgers.edu}

\date{This version: February 1, 2018, incorporating final galley proof corrections. To appear in \emph{Proceedings of the American Mathematical Society}.}

\begin{abstract}
We extend the well-known Sacks--Uhlenbeck energy gap result \cite[Theorem 3.3]{Sacks_Uhlenbeck_1981} for harmonic maps from closed Riemann surfaces into closed Riemannian manifolds from the case of maps with small energy (thus near a constant map), to the case of harmonic maps with high absolute energy but small energy relative to a reference harmonic map.
\end{abstract}




\thanks{Paul Feehan was partially supported by National Science Foundation grant DMS-1510064 and the Oswald Veblen Fund and Fund for Mathematics (Institute for Advanced Study, Princeton) during the preparation of this article.}

\maketitle
\tableofcontents

\section{Introduction}
Let $(M,g)$ and $(N,h)$ be a pair of smooth Riemannian manifolds, with $M$ orientable. One defines the \emph{harmonic map energy function} by
\begin{equation}
\label{eq:Harmonic_map_energy_functional}
\sE_{g,h}(f)
:=
\frac{1}{2} \int_M |df|_{g,h}^2 \,d\vol_g,
\end{equation}
for smooth maps, $f:M\to N$, where $df:TM \to TN$ is the differential map. A map $f \in C^\infty(M;N)$ is \emph{(weakly) harmonic} if it is a \emph{critical point} of $\sE_{g,h}$, so $\sE_{g,h}'(f)(u) = 0$ for all $u \in C^\infty(M;f^*TN)$, where
\[
\sE_{g,h}'(f)(u) = \int_M \langle df, u\rangle_{g,h} \,d\vol_g.
\]
The purpose of this article is to prove

\begin{mainthm}[Relative energy gap for harmonic maps of Riemann surfaces into real analytic Riemannian manifolds]
\label{mainthm:Relative_energy_gap_harmonic_maps}
Let $(M,g)$ be a closed Riemann surface and $(N,h)$ a closed, real analytic Riemannian manifold equipped with a real analytic isometric embedding into a Euclidean space, $\RR^n$. If $f_\infty \in C^\infty(M; N)$ is a harmonic map, then there is a constant $\eps = \eps(f_\infty,g,h) \in (0,1]$ with the following significance. If $f \in C^\infty(M; N)$ is a harmonic map obeying
\begin{equation}
\label{eq:Relative_harmonic_map_energy_bound}
\|d(f - f_\infty)\|_{L^2(M;\RR^n)} + \|f - f_\infty\|_{L^2(M;\RR^n)}  < \eps,
\end{equation}
then $\sE_{g,h}(f) = \sE_{g,h}(f_\infty)$.
\end{mainthm}

\begin{rmk}[Generalizations to the case of harmonic maps with potentials]
\label{rmk:Harmonic_maps_with_potentials}
In physics, harmonic maps arise in the context of non-linear sigma models and with such applications in mind, Theorem \ref{mainthm:Relative_energy_gap_harmonic_maps} should admit generalizations to allow, for example, the addition to $\sE_{g,h}$ of a real analytic potential function, $V:C^\infty(M;N) \to \RR$, in the definition \eqref{eq:Harmonic_map_energy_functional} of the energy, as explored by Branding \cite{Branding_2016agag}.
\end{rmk}

Naturally, Theorem \ref{mainthm:Relative_energy_gap_harmonic_maps} continues to hold if the condition \eqref{eq:Relative_harmonic_map_energy_bound} is replaced by the stronger (and conformally invariant) hypothesis,
\[
\|d(f - f_\infty)\|_{L^2(M;\RR^n)} + \|f - f_\infty\|_{L^\infty(M;\RR^n)}  < \eps.
\]
Thus, if $\wp$ is any conformal diffeomorphism of $(M,g)$, then the preceding condition on $f, f_\infty$ holds if and only if the harmonic maps $f\circ\wp, f_\infty\circ\wp$ obey
\[
\|d(f\circ\wp - f_\infty\circ\wp)\|_{L^2(M;\RR^n)} + \|f\circ\wp - f_\infty\circ\wp\|_{L^\infty(M;\RR^n)}  < \eps.
\]
Hence, the constants, $Z,\sigma,\theta$ in Theorem \ref{mainthm:Relative_energy_gap_harmonic_maps} are in this sense independent of the action of the conformal group of $(M,g)$ on harmonic maps from $M$ to $N$.

Theorem \ref{mainthm:Relative_energy_gap_harmonic_maps} may be viewed, in part, as a generalization of the following energy gap result due to Sacks and Uhlenbeck and who do not require that the target manifold be real analytic.

\begin{thm}[Energy gap near the constant map]
\label{thm:Energy_gap_near_constant_map}
\cite[Theorem 3.3]{Sacks_Uhlenbeck_1981}
Let $(M,g)$ be a closed Riemann surface and $(N,h)$ be a closed, smooth Riemannian manifold. Then there is a constant, $\eps > 0$, such that if $f\in C^\infty(M;N)$ is harmonic and $\sE_{g,h}(f) < \eps$, then $f$ is a constant map and $\sE_{g,h}(f) = 0$.
\end{thm}

The Sacks--Uhlenbeck Energy Gap Theorem \ref{thm:Energy_gap_near_constant_map} has been generalized by Branding \cite[Lemma 4.9]{Branding_2015} and by Jost and his collaborators \cite[Proposition 4.2]{Chen_Jost_Li_Wang_2006}, \cite[Proposition 5.2]{Jost_Liu_Zhu_2015preprint} to the case of Dirac-harmonic pairs. Theorem \ref{thm:Energy_gap_near_constant_map} ensures positivity of the constant $\hbar$ in the

\begin{defn}[Dirac--Planck constant]
\label{defn:Planck_constant}
Let $(N,h)$ be a closed, smooth Riemannian manifold. Then $\hbar$ denotes the \emph{least energy} of a non-constant $C^\infty$ map from $(S^2,g_\round)$ into $(N,h)$, where $g_\round$ is the standard round metric of radius one on $S^2$.
\end{defn}

The energy gap near the `ground state' characterized by the constant maps from $(S^2,g_\round)$ to $(N,h)$ appears to be unusual in the light of the following counter-example due to Li and Wang \cite{Li_Wang_2015} when $(N,h)$ is only $C^\infty$ rather than real analytic.

\begin{exmp}[Non-discreteness of the energy spectrum for harmonic maps from $S^2$ into a smooth Riemannian manifold with boundary]
\label{exmp:Li_Wang_section_4}
(See \cite[Section 4]{Li_Wang_2015}.)
There exists a smooth Riemannian metric $h$ on $N = S^2\times (-1,1)$ such that the energies of harmonic maps from $(S^2,g_\round)$ to $(N,h)$ have an accumulation point at the energy level $4\pi$, where, $g_\round$ denotes the standard round metric of radius one.
\end{exmp}

Thus we would \emph{not} expect Theorem \ref{mainthm:Relative_energy_gap_harmonic_maps} to hold when the hypothesis that $(N,h)$ is real analytic is omitted, except for the case where $f_\infty$ is a constant map. On the other hand, when $(N,h)$ \emph{is} real analytic, one has the following conjecture due to Lin \cite{Lin_1999}.

\begin{conj}[Discreteness for energies of harmonic maps from closed Riemann surfaces into analytic closed Riemannian manifolds]
\label{conj:Discreteness_harmonic_map_energies}
(Lin \cite[Conjecture 5.7]{Lin_1999}.)
Assume the hypotheses of Theorem \ref{mainthm:Relative_energy_gap_harmonic_maps} and that $(M,g)$ is the two-sphere, $S^2$, with its standard, round metric. Then the subset of critical values of the energy function, $\sE_{g,h}:C^\infty(S^2;N) \to [0,\infty)$, is closed and discrete.
\end{conj}

One may therefore view Theorem \ref{mainthm:Relative_energy_gap_harmonic_maps} as supporting evidence of the validity of Conjecture \ref{conj:Discreteness_harmonic_map_energies}. In the special case that $N$ is the Lie group $\U(n)$ with $n \geq 2$ and its standard Riemannian metric, Valli \cite[Corollary 8]{Valli_1988} has shown (using ideas of Uhlenbeck \cite{Uhlenbeck_1989}) that the energies of harmonic maps from $(S^2,g_\round)$ into $\U(n)$ are integral multiples of $8\pi$. If $(N,h)$ has non-positive curvature sectional curvature, then Adachi and Sunada \cite[Theorem 1]{Adachi_Sunada_1985} have shown that Conjecture \ref{conj:Discreteness_harmonic_map_energies} holds when $(M,g)$ is any closed Riemann surface.

\subsection{Outline of the article}
\label{subsec:Outline}
In Section \ref{sec:Lojasiewicz-Simon_gradient_inequality_harmonic_map_function}, we review the {\L}ojasiewicz--Simon gradient inequality for the harmonic map energy function based on results of the author and Maridakis \cite{Feehan_Maridakis_Lojasiewicz-Simon_harmonic_maps_v6} and Simon \cite{Simon_1983, Simon_1985}. In Section \ref{sec:Apriori_estimate_difference_two_harmonic_maps}, we prove certain \apriori estimates for the difference of two harmonic maps and, with the aid of the {\L}ojasiewicz--Simon gradient inequality, complete the proof of Theorem \ref{mainthm:Relative_energy_gap_harmonic_maps}.

\subsection{Acknowledgments}
\label{subsec:Acknowledgments}
I am very grateful to the Institute for Advanced Study, Princeton, for their support during the preparation of this article. I thank Sagun Chanillo, Fernando Cod{\'a} Marques, Joel Hass, Tobias Lamm, Paul Larain, Fang-Hua Lin, Thomas Parker, Tristan Rivi{\'e}re, Michael Taylor, Peter Topping, Karen Uhlenbeck, and especially Manousos Maridakis for helpful questions and comments that influenced the development of this article.

\section{{\L}ojasiewicz--Simon gradient inequality for the harmonic map energy function}
\label{sec:Lojasiewicz-Simon_gradient_inequality_harmonic_map_function}
In this section, we closely follow our treatment of the {\L}ojasiewicz--Simon gradient inequality for abstract and harmonic map energy functions provided by the author and Maridakis in \cite[Sections 1.1, 1.2, and 1.4]{Feehan_Maridakis_Lojasiewicz-Simon_harmonic_maps_v6}. Useful references for harmonic maps include Eells and Lemaire \cite{Eells_Lemaire_1978, Eells_Lemaire_1988}, Hamilton \cite{Hamilton_1975}, H{\'e}lein \cite{Helein_harmonic_maps}, H{\'e}lein and Wood \cite{Helein_Wood_2008}, Jost
\cite{Jost_harmonic_maps_surfaces, Jost_harmonic_maps_riemannian_manifolds, Jost_two_dim_geom_var_probs, Jost_1993pspm, Jost_riemannian_geometry_geometric_analysis}, Moser \cite{Moser_2005}, Parker \cite{ParkerHarmonic}, Sacks and Uhlenbeck \cite{Sacks_Uhlenbeck_1981, Sacks_Uhlenbeck_1982}, Schoen \cite{Schoen_1984msri}, Simon \cite{Simon_1996}, Struwe \cite{Struwe_variational_methods}, Urakawa \cite{Urakawa_1993}, Xin \cite{Xin_Geometry_harmonic_maps}, and citations contained therein.

When clear from the context, we omit explicit mention of the Riemannian metrics $g$ on $M$ and $h$ on $N$ and write $\sE = \sE_{g,h}$. Although initially defined for smooth maps, the energy function $\sE$ in \eqref{eq:Harmonic_map_energy_functional}, extends to the case of Sobolev maps of class $W^{1,2}$. To define the gradient, $\sM = \sM_{g,h}$, of the energy function $\sE$ in \eqref{eq:Harmonic_map_energy_functional} with respect to the $L^2$ metric on $C^\infty(M;N)$, we first choose an isometric embedding, $(N,h) \subset \RR^n$ for a sufficiently large $n$ (courtesy of the isometric embedding theorem due to Nash \cite{Nash_1956}), and recall that\footnote{Compare \cite[Equations (8.1.10) and (8.1.13)]{Jost_riemannian_geometry_geometric_analysis}, where Jost uses variations of $f$ of the form $\exp_{f}(tu)$.} by \cite[Equations (2.2)(i) and (ii)]{Simon_1996}
\begin{align*}
\left(u, \sM(f)\right)_{L^2(M,g)}
&:= \sE'(f)(u) = \left.\frac{d}{dt}\sE(\pi(f + tu))\right|_{t=0}
\\
&\,= \left(u, \Delta_g f\right)_{L^2(M,g)}
\\
&\,= \left(u, d\pi_h(f)\Delta_g f\right)_{L^2(M,g)},
\end{align*}
for all $u \in C^\infty(M;f^*TN)$, where $\pi_h$ is the nearest point projection onto $N$ from a normal tubular neighborhood and $d\pi_h(y):\RR^n \to T_yN$ is orthogonal projection, for all $y \in N$.  By \cite[Lemma 1.2.4]{Helein_harmonic_maps}, we have
\begin{equation}
\label{eq:Gradient_harmonic_map_operator}
\sM(f) = d\pi_h(f)\Delta_g f = \Delta_g f - A_h(f)(df,df),
\end{equation}
as in \cite[Equations (2.2)(iii) and (iv)]{Simon_1996}. Here, $A_h$ denotes the second fundamental form of the isometric embedding, $(N,h) \subset \RR^n$, and
\begin{equation}
\label{eq:Laplace-Beltrami_operator}
\Delta_g
:=
-\divg_g \grad_g
=
d^{*,g}d
=
-\frac{1}{\sqrt{\det g}} \frac{\partial}{\partial x^\beta}
\left(\sqrt{\det g}\, \frac{\partial f}{\partial x^\alpha} \right)
\end{equation}
denotes the Laplace-Beltrami operator for $(M,g)$ (with the opposite sign convention to that of \cite[Equations (1.14) and (1.33)]{Chavel}) acting on the scalar components $f^i$ of $f = (f^1,\ldots,f^n)$ and the $\{x^\alpha\}$ denote local coordinates on $M$. As usual, the gradient vector field, $\grad_g f^i \in C^\infty(TM)$, is defined by $\langle\grad_g f^i,\xi\rangle_g := df^i(\xi)$ for all $\xi \in C^\infty(TM)$ and $1
\leq i \leq n$ and the divergence function, $\divg_g\xi \in C^\infty(M;\RR)$, by the pointwise trace,
$\divg_g\xi := \tr(\eta \mapsto \nabla_\xi^g\eta)$, for all $\eta \in C^\infty(TM)$.

Given a smooth map, $f:M\to N$, an isometric embedding, $(N,h) \subset \RR^n$, a non-negative integer $k$, and constant $p \in [1,\infty)$, we define the Sobolev norms,
\[
\|f\|_{W^{k,p}(M;\RR^n)} := \left(\sum_{i=1}^n \|f^i\|_{W^{k,p}(M;\RR^n)}^p\right)^{1/p},
\]
with
\[
\|f^i\|_{W^{k,p}(M;\RR^n)} := \left(\sum_{j=0}^k \int_M |(\nabla^g)^j f^i|^p \,d\vol_g\right)^{1/p},
\]
where $\nabla^g$ denotes the Levi-Civita connection on $TM$ and all associated bundles (that is, $T^*M$ and their tensor products), and if $p = \infty$, we define
\[
\|f\|_{W^{k,\infty}(M;\RR^n)}
:=
\sum_{i=1}^{n}\sum_{j=0}^k \esssup_M |(\nabla^g)^j f^i|.
\]
If $k=0$, then we denote $\|f\|_{W^{0,p}(M;\RR^n)} = \|f\|_{L^p(M;\RR^n)}$. For $p \in [1,\infty)$ and nonnegative integers $k$, we use \cite[Theorem 3.12]{AdamsFournier} (applied to $W^{k,p}(M;\RR^n)$ and noting that $M$ is a closed manifold) and Banach space duality to define
\[
W^{-k,p'}(M;\RR^n) := \left(W^{k,p}(M;\RR^n)\right)^*,
\]
where $p'\in (1,\infty]$ is the dual exponent defined by $1/p+1/p'=1$. Elements of the continuous Banach dual space, $(W^{k,p}(M;\RR^n))^*$, may be characterized via \cite[Section 3.10]{AdamsFournier} as distributions in the Schwartz space, $\sD'(M;\RR^n)$ \cite[Section 1.57]{AdamsFournier}.

Spaces of H\"older continuous maps, $C^{k,\lambda}(M;N)$ for $\lambda \in (0,1)$ and integers $k \geq 0$, and norms,
\[
\|f\|_{C^{k,\lambda}(M;\RR^n)},
\]
may be defined as in \cite[Section 1.29]{AdamsFournier}.

We note that if $(N,h)$ is real analytic, then the isometric embedding, $(N,h) \subset \RR^n$, may also be chosen to be analytic by the analytic isometric embedding theorem due to Nash \cite{Nash_1966}, with a simplified proof due to Greene and Jacobowitz \cite{Greene_Jacobowitz_1971}).

\begin{defn}[Harmonic map]
\label{defn:Harmonic_map}
(See \cite[Definition 1.4.9]{Helein_harmonic_maps}.)
A map $f \in W^{1,2}(M;N)$ is called \emph{weakly harmonic} if it is a critical point of the $L^2$-energy functional \eqref{eq:Harmonic_map_energy_functional}, that is
\[
\sE'(f)(u) = 0, \quad \forall\, u\in C^\infty(M; f^*TN),
\]
and a map $f \in W^{2,p}(M;N)$, for $p \in [1,\infty]$, is called \emph{harmonic} if
\begin{equation}
\label{eq:Harmonic_map_equation}
\Delta_g f - A_h(df,df) = 0 \quad\text{a.e. on } M.
\end{equation}
\end{defn}

A well-known result due to H\'elein \cite[Theorem 4.1.1]{Helein_harmonic_maps} tells us that if $M$ has dimension $d=2$, then $f \in C^\infty(M;N)$; for $d \geq 3$, regularity results are far more limited --- see, for example,
\cite[Theorem 4.3.1]{Helein_harmonic_maps} due to Bethuel. From \cite{Feehan_Maridakis_Lojasiewicz-Simon_harmonic_maps_v6}, we recall the

\begin{thm}[{\L}ojasiewicz--Simon gradient inequality for the energy function for maps between pairs of Riemannian manifolds]
\label{thm:Lojasiewicz-Simon_gradient_inequality_energy_functional_Sobolev}
(See \cite[Theorem 5]{Feehan_Maridakis_Lojasiewicz-Simon_harmonic_maps_v6}.)
Let $d\geq 2$ and $k \geq 1$ be integers and $p\in (1,\infty)$ be such that $kp > d$. Let $(M,g)$ and $(N,h)$ be closed, smooth Riemannian manifolds, with $M$ of dimension $d$. If $(N,h)$ is real analytic (respectively, $C^\infty$) and $f\in W^{k,p}(M;N)$, then the gradient map $\sM$ in \eqref{eq:Gradient_harmonic_map_operator} for the energy function, $\sE:W^{k,p}(M; N)\to\RR$, in \eqref{eq:Harmonic_map_energy_functional},
\[
W^{k,p}(M; N) \ni f \mapsto \sM(f) \in W^{k-2,p}(M; f^*TN) \subset W^{k-2,p}(M; \RR^n),
\]
is a real analytic (respectively, $C^\infty$) map of Banach spaces. If $(N,h)$ is real analytic and $f_\infty \in W^{k,p}(M; N)$ is a weakly harmonic map, then there are positive constants $Z \in (0, \infty)$, and $\sigma \in (0,1]$, and $\theta \in [1/2,1)$, depending on $f_\infty$, $g$, $h$, $k$, $p$, with the following significance. If $f\in W^{k,p}(M;N)$ obeys
\begin{equation}
\label{eq:Lojasiewicz-Simon_gradient_inequality_harmonic_map_neighborhood_Sobolev}
\|f - f_\infty\|_{W^{k,p}(M;\RR^n)} < \sigma,
\end{equation}
then the gradient $\sM$ in \eqref{eq:Gradient_harmonic_map_operator} of the harmonic map energy function $\sE$ in
\eqref{eq:Harmonic_map_energy_functional} obeys
\begin{equation}
\label{eq:Lojasiewicz-Simon_gradient_inequality_harmonic_map_energy_functional_Sobolev}
\|\sM(f)\|_{W^{k-2,p}(M;f^*TN)}
\geq
Z|\sE(f) - \sE(f_\infty)|^\theta.
\end{equation}
\end{thm}

\begin{rmk}[On the hypotheses of Theorem \ref{thm:Lojasiewicz-Simon_gradient_inequality_energy_functional_Sobolev}]
When $k=d$ and $p=1$, then $W^{d,1}(M;\RR) \subset C(M;\RR)$ is a continuous embedding by \cite[Theorem 4.12]{AdamsFournier} and $W^{d,1}(M;\RR)$ is a Banach algebra by \cite[Theorem 4.39]{AdamsFournier}. In particular, $W^{d,1}(M; N)$ is a real analytic Banach manifold by \cite[Proposition 3.2]{Feehan_Maridakis_Lojasiewicz-Simon_harmonic_maps_v6} and the harmonic map energy function, $\sE: W^{d,1}(M; N) \to \RR$, is real analytic by
\cite[Proposition 3.5]{Feehan_Maridakis_Lojasiewicz-Simon_harmonic_maps_v6}. However, the operator $\sM'(f_\infty):W^{d,1}(M; f_\infty^*TN) \to W^{d-2,1}(M; f_\infty^*TN)$ may not be Fredholm.
\end{rmk}

Theorem \ref{thm:Lojasiewicz-Simon_gradient_inequality_energy_functional_Sobolev} extends a version of the {\L}ojasiewicz--Simon gradient inequality that is stated by Simon as \cite[Equation (4.27)]{Simon_1985} and can be derived from his more general \cite[Theorem 3]{Simon_1983}.

\begin{thm}[{\L}ojasiewicz--Simon gradient inequality for the energy function for maps between pairs of Riemannian manifolds]
\label{thm:Lojasiewicz-Simon_gradient_inequality_energy_functional_Holder}
(See \cite[Corollary 6]{Feehan_Maridakis_Lojasiewicz-Simon_harmonic_maps_v6}, \cite[Theorem 3]{Simon_1983}, \cite[Equation (4.27)]{Simon_1985}.)
Let $d\geq 2$ and $\lambda\in (0,1)$ be constants, $(M,g)$ a closed, smooth Riemannian manifold of dimension $d$ and $(N,h)$ is closed, real analytic Riemannian manifold. If $f_\infty \in C^{2,\lambda}(M; N)$ is a harmonic map, then there are positive constants $Z \in (0, \infty)$, and $\sigma \in (0,1]$, and $\theta \in [1/2,1)$, depending on $f_\infty$, $g$, $h$, $\lambda$, with the following significance. If $f\in C^{2,\lambda}(M;N)$ obeys
\begin{equation}
\label{eq:Lojasiewicz-Simon_gradient_inequality_harmonic_map_neighborhood_Holder}
\|f - f_\infty\|_{C^{2,\lambda}(M;\RR^n)} < \sigma,
\end{equation}
then the gradient $\sM$ in \eqref{eq:Gradient_harmonic_map_operator} of the harmonic map energy function $\sE$ in
\eqref{eq:Harmonic_map_energy_functional} obeys
\begin{equation}
\label{eq:Lojasiewicz-Simon_gradient_inequality_harmonic_map_energy_functional_Holder}
\|\sM(f)\|_{L^2(M;f^*TN)}
\geq
Z|\sE(f) - \sE(f_\infty)|^\theta.
\end{equation}
\end{thm}

\begin{rmk}[Other versions of the {\L}ojasiewicz--Simon gradient inequality for the harmonic map energy function]
Topping \cite[Lemma 1]{Topping_1997} proved a {\L}ojasiewicz-type gradient inequality for maps, $f:S^2 \to S^2$, with small energy, with the latter criterion replacing the usual small $C^{2,\lambda}(M;\RR^n)$ norm criterion of Simon for the difference between a map and a critical point. Topping's result is generalized by Liu and Yang in \cite[Lemma 3.3]{Liu_Yang_2010}. Kwon \cite[Theorem 4.2]{KwonThesis} obtains a {\L}ojasiewicz-type gradient inequality for maps, $f:S^2 \to N$, that are $W^{2,p}(S^2;\RR^n)$-close to a harmonic map, with $1 < p \leq 2$.
\end{rmk}

When $d=2$ in the hypotheses of Theorem \ref{thm:Lojasiewicz-Simon_gradient_inequality_energy_functional_Sobolev}, the reader will note that the two cases that are most directly applicable to a proof of Theorem \ref{mainthm:Relative_energy_gap_harmonic_maps} are omitted, namely the cases $k=2$ and $p=1$ or $k=1$ and $p=2$, which are both critical since $kp=d$. We shall briefly comment on each of these two cases.

When $d=2$, $k=1$, and $p=2$, it appears very difficult to verify the hypotheses of
\cite[Theorem 2]{Feehan_Maridakis_Lojasiewicz-Simon_harmonic_maps_v6}. The analytical difficulties are very much akin to those confronted by H\'elein \cite{Helein_harmonic_maps} in his celebrated proof of smoothness of weakly harmonic maps from Riemann surfaces. However, it is unclear that H\'elein's methods could be used to extend Theorem \ref{thm:Lojasiewicz-Simon_gradient_inequality_energy_functional_Sobolev} to the case $d=2$, $k=1$, and $p=2$.

Similarly, when $d=2$, $k=2$, and $p=1$, it is very difficult to verify the hypotheses of
\cite[Theorem 2]{Feehan_Maridakis_Lojasiewicz-Simon_harmonic_maps_v6}. One might speculate that a version of Theorem \ref{thm:Lojasiewicz-Simon_gradient_inequality_energy_functional_Sobolev} could hold if the role of the pair of Sobolev spaces, $W^{2,1}(M; f_\infty^*TN)$ and $L^1(M; f_\infty^*TN)$, were replaced by suitably defined \emph{local Hardy spaces}. We refer the reader to Semmes \cite{Semmes_1994}, Stein \cite{Stein_HarmonicAnalysis}, and Taylor \cite{Taylor_tools_for_pde} for introductions to Hardy spaces of functions on Euclidean space and to H\'elein \cite{Helein_harmonic_maps} for their application to the problem of regularity for weakly harmonic maps from Riemann surfaces. Auscher, McIntosh, Morris \cite{Auscher_McIntosh_Morris_2015}, Carbonaro, McIntosh, and Morris \cite{Carbonaro_McIntosh_Morris_2013} and Taylor \cite{Taylor_2009} provide definitions of local Hardy spaces on Riemannian manifolds. However, the analytical difficulties appear formidable in any such approach.

Fortunately, in our proof of Theorem \ref{mainthm:Relative_energy_gap_harmonic_maps}, we can apply Theorem \ref{thm:Lojasiewicz-Simon_gradient_inequality_energy_functional_Sobolev} with \emph{non-critical} Sobolev exponents, namely $d=2$, $k=2$, and $p\in(1,\infty)$ by exploiting certain \apriori estimates for harmonic maps similar to those used by Sacks and Uhlenbeck \cite{Sacks_Uhlenbeck_1981}.

Theorem \ref{thm:Lojasiewicz-Simon_gradient_inequality_energy_functional_Sobolev} is proved by the author and Maridakis in \cite{Feehan_Maridakis_Lojasiewicz-Simon_harmonic_maps_v6} as a consequence of a more general abstract {\L}ojasiewicz--Simon gradient inequality for an analytic function on a Banach space, namely
\cite[Theorem 2]{Feehan_Maridakis_Lojasiewicz-Simon_harmonic_maps_v6}, while Theorem \ref{thm:Lojasiewicz-Simon_gradient_inequality_energy_functional_Holder} may be deduced as a consequence of \cite[Theorem 2]{Feehan_Maridakis_Lojasiewicz-Simon_harmonic_maps_v6}.

To state the abstract \cite[Theorem 2]{Feehan_Maridakis_Lojasiewicz-Simon_harmonic_maps_v6}, we let $\sX$ be a Banach space and let $\sX^*$ denote its continuous dual space. We call a bilinear form, $b:\sX\times\sX \to \RR$, \emph{definite} if $b(x,x) \neq 0$ for all $x \in \sX\less\{0\}$. We say that a continuous \emph{embedding} of a Banach space into its continuous dual space, $\jmath:\sX\to\sX^*$, is \emph{definite} if the pullback of the canonical pairing, $\sX\times\sX \ni (x,y) \mapsto \langle x,\jmath(y)\rangle_{\sX\times\sX^*} \to \RR$, is a definite bilinear form. (This hypothesis on the continuous embedding, $\sX \subset \sX^*$, is easily achieved given a continuous embedding of $\sX$ into a Hilbert space $\sH$ but the increased generality is often convenient.)

\begin{defn}[Gradient map]
\label{defn:Huang_2-1-1}
(See \cite[Section 2.5]{Berger_1977}, \cite[Definition 2.1.1]{Huang_2006}.)
Let $\sU\subset \sX$ be an open subset of a Banach space, $\sX$, and let $\tilde\sX$ be a Banach space with continuous embedding, $\tilde\sX \subseteqq \sX^*$. A continuous map, $\sM:\sU\to \tilde\sX$, is called a \emph{gradient map} if there exists a $C^1$ function, $\sE:\sU\to\RR$, such that
\begin{equation}
\label{eq:Differential_and_gradient_maps}
\sE'(x)v = \langle v,\sM(x)\rangle_{\sX\times\sX^*}, \quad \forall\, x \in \sU, \quad v \in \sX,
\end{equation}
where $\langle \cdot , \cdot \rangle_{\sX\times\sX^*}$ is the canonical bilinear form on $\sX\times\sX^*$. The real-valued function, $\sE$, is called a \emph{potential} for the gradient map, $\sM$.
\end{defn}

\begin{thm}[{\L}ojasiewicz--Simon gradient inequality for analytic functions on Banach spaces]
\label{thm:Lojasiewicz-Simon_gradient_inequality}
(See \cite[Corollary 3]{Feehan_Maridakis_Lojasiewicz-Simon_harmonic_maps_v6}.)
Let $\sX$ and $\tilde\sX$ be Banach spaces with continuous embeddings, $\sX \subset \tilde\sX \subset \sX^*$, and such that the embedding, $\sX \subset \sX^*$, is definite. Let $\sU \subset \sX$ be an open subset, $\sE:\sU\to\RR$ a $C^2$ function with real analytic gradient map, $\sM:\sU\to\tilde\sX$, and $x_\infty\in\sU$ a critical point of $\sE$, that is, $\sM(x_\infty) = 0$. If $\sM'(x_\infty):\sX\to \tilde\sX$ is a Fredholm operator with index zero, then there are constants, $Z \in (0,\infty)$, and $\sigma \in (0,1]$, and $\theta \in [1/2, 1)$, with the following significance. If $x \in \sU$ obeys
\begin{equation}
\label{eq:Lojasiewicz-Simon_gradient_inequality_neighborhood_general}
\|x-x_\infty\|_\sX < \sigma,
\end{equation}
then
\begin{equation}
\label{eq:Lojasiewicz-Simon_gradient_inequality_analytic_functional_general}
\|\sM(x)\|_{\tilde\sX} \geq Z|\sE(x) - \sE(x_\infty)|^\theta.
\end{equation}
\end{thm}

Theorem \ref{thm:Lojasiewicz-Simon_gradient_inequality_energy_functional_Sobolev} follows from Theorem \ref{thm:Lojasiewicz-Simon_gradient_inequality} by choosing
\[
\sX = W^{k,p}(M;f_\infty^*TN) \quad\text{and}\quad \tilde\sX = W^{k-2,p}(M;f_\infty^*TN).
\]
Theorem \ref{thm:Lojasiewicz-Simon_gradient_inequality_energy_functional_Holder} follows from Theorem \ref{thm:Lojasiewicz-Simon_gradient_inequality_energy_functional_Sobolev} when one can choose $k\geq 1$ and $p \in (1,\infty)$ with $kp>d$ so that there are continuous Sobolev embeddings, $C^{2,\lambda}(M;f_\infty^*TN) \subset W^{k,p}(M;f_\infty^*TN)$, and $L^2(M;f_\infty^*TN) \subset W^{k-2,p}(M;f_\infty^*TN)$. For example, if $d=2$, then $k=p=2$ will do. For $d\geq 2$, we may choose $k=1$ and $d<p<\infty$ provided $L^2(M;\RR) \subset W^{-1,p}(M;\RR)$ is a continuous embedding or, equivalently, $W^{1,p'}(M;\RR) \subset L^2(M;\RR)$ is a continuous embedding, where $p'=p/(p-1)$. According to \cite[Theorem 4.12]{AdamsFournier} when $1\leq p'<d$, the latter embedding is continuous if $(p')^* = dp'/(d-p') = dp/(d(p-1)-p) \geq 2$. But we must choose $p>d$ when $k=1$ in Theorem \ref{thm:Lojasiewicz-Simon_gradient_inequality_energy_functional_Sobolev} and if $p=d$, then $dp/(d(p-1)-p) = d^2/(d(d-1)-d) = d/(d-2) \geq 2$ implies $d \geq 2d-4$ or $d \leq 4$. Hence, Theorem \ref{thm:Lojasiewicz-Simon_gradient_inequality_energy_functional_Sobolev} implies Theorem \ref{thm:Lojasiewicz-Simon_gradient_inequality_energy_functional_Holder} when $d=2,3$ (the case $d=4$ is excluded since $p>d$ leads to $d<4$ in the preceding inequalities). For arbitrary $d \geq 2$, another abstract {\L}ojasiewicz--Simon gradient inequality \cite[Theorem 2.4.2 (i)]{Huang_2006} due to Huang implies Theorem \ref{thm:Lojasiewicz-Simon_gradient_inequality_energy_functional_Holder} with the choices
\[
\sX = C^{2,\lambda}(M;f_\infty^*TN), \quad \tilde\sX = C^\lambda(M;f_\infty^*TN), \quad \quad \sH = L^2(M;f_\infty^*TN),
\]
and $\sH_\sA = W^{2,2}(M;f_\infty^*TN)$ with $\sA = \Delta_g + 1$ in \cite[Hypotheses (H1)--(H3), pages 34--35]{Huang_2006}. We refer the reader to \cite{Feehan_Maridakis_Lojasiewicz-Simon_coupled_Yang-Mills_v4} for an exposition of Huang's \cite[Theorem 2.4.2 (i)]{Huang_2006}.

Alternatively, our \cite[Theorem 3]{Feehan_Maridakis_Lojasiewicz-Simon_harmonic_maps_v6} implies Theorem \ref{thm:Lojasiewicz-Simon_gradient_inequality_energy_functional_Holder}, as we show in the proof of \cite[Corollary 6]{Feehan_Maridakis_Lojasiewicz-Simon_harmonic_maps_v6}.

\section{\Apriori estimate for the difference of two harmonic maps}
\label{sec:Apriori_estimate_difference_two_harmonic_maps}
In this section, we give two proofs of Theorem \ref{mainthm:Relative_energy_gap_harmonic_maps}, based on Theorems \ref{thm:Lojasiewicz-Simon_gradient_inequality_energy_functional_Sobolev} and \ref{thm:Lojasiewicz-Simon_gradient_inequality_energy_functional_Holder}, respectively. We begin with the

\begin{lem}[\Apriori $W^{2,p}$ estimate for the difference of two harmonic maps]
\label{lem:Apriori_W2p_estimate_difference_two_harmonic_maps}
Let $(M,g)$ be a closed Riemann surface, $(N,h)$ a closed, smooth Riemannian manifold, and $p \in (1, 2]$ a constant. Then there is a constant $C=C(g,h,p) \in [1,\infty)$ with the following significance. If $f, f_\infty \in C^\infty(M; N)$ are harmonic maps and $q = 2p/(2-p) \in (2,\infty]$, then
\begin{equation}
\label{eq:Apriori_W2p_estimate_difference_two_harmonic_maps}
\|f-f_\infty\|_{W^{2,p}(M;\RR^n)} \leq C\left(\|df\|_{L^q(M;\RR^n)} + \|df_\infty\|_{L^q(M;\RR^n)} + 1\right)\|f - f_\infty\|_{W^{1,2}(M;\RR^n)}.
\end{equation}
\end{lem}

\begin{proof}
Because $f$ and $f_\infty$ are harmonic, equation \eqref{eq:Harmonic_map_equation} implies that
\begin{align*}
\Delta_g f - A_h(df,df) &= 0,
\\
\Delta_g f_\infty - A_h(df_\infty,df_\infty) &= 0,
\end{align*}
and therefore,
\begin{equation}
\label{eq:Harmonic_map_difference_equation}
\Delta_g (f-f_\infty) - A_h(df,d(f-f_\infty)) - A_h(d(f-f_\infty),df_\infty) = 0.
\end{equation}
Because $1/p = 1/2 + 1/q$ by hypothesis, the preceding equality yields the estimate,
\[
\|\Delta_g (f-f_\infty)\|_{L^p(M;\RR^n)} \leq C\left(\|df\|_{L^q(M;\RR^n)} + \|df_\infty\|_{L^q(M;\RR^n)} \right) \|d(f-f_\infty)\|_{L^2(M;\RR^n)},
\]
with $C = C(h) \in [1,\infty)$. The standard \apriori $W^{2,p}$ estimate for an elliptic, linear, scalar, second-order partial differential operator over a bounded domain in Euclidean space \cite[Theorem 9.13]{GilbargTrudinger} yields the bound,
\[
\|f-f_\infty\|_{W^{2,p}(M;\RR^n)} \leq C\left(\|\Delta_g (f-f_\infty)\|_{L^p(M;\RR^n)} + \|f-f_\infty\|_{L^p(M;\RR^n)} \right),
\]
for a constant $C=C(g,p) \in [1,\infty)$. Combining the preceding two inequalities gives
\begin{multline*}
\|f-f_\infty\|_{W^{2,p}(M;\RR^n)}
\leq
C\left(\|df\|_{L^q(M;\RR^n)} + \|df_\infty\|_{L^q(M;\RR^n)} \right) \|d(f-f_\infty)\|_{L^2(M;\RR^n)}
\\
+ C\|f-f_\infty\|_{L^p(M;\RR^n)},
\end{multline*}
for $C=C(g,h,p) \in [1,\infty)$. Since $p\leq 2$, this yields the desired estimate.
\end{proof}

\begin{lem}[\Apriori $W^{1,q}$ estimate for a harmonic map]
\label{lem:Apriori_W1q_estimate_harmonic_map}
Let $(M,g)$ be a closed Riemann surface, $(N,h)$ a closed, smooth Riemannian manifold, and $q \in (2,\infty)$ a constant. Then there is a constant $\eps=\eps(g,h,q) \in (0,1]$ with the following significance. If $f, f_\infty \in C^\infty(M; N)$ are harmonic maps obeying \eqref{eq:Relative_harmonic_map_energy_bound}, then
\begin{equation}
\label{eq:Apriori_W1q_estimate_harmonic_map}
\|f\|_{W^{1,q}(M;\RR^n)} \leq 1 + 3\|f_\infty\|_{W^{1,q}(M;\RR^n)}.
\end{equation}
\end{lem}

\begin{proof}
For $p\in (1,2)$ defined by $p^* := 2p/(2-p) = q$, we observe that $W^{1,p}(M;\RR) \subset L^q(M;\RR)$ is a continuous Sobolev embedding by \cite[Theorem 4.12]{AdamsFournier}. Hence, the estimate \eqref{eq:Apriori_W2p_estimate_difference_two_harmonic_maps} yields
\begin{align*}
\|f-f_\infty\|_{W^{1,q}(M;\RR^n)}
&\leq
C\|f-f_\infty\|_{W^{2,p}(M;\RR^n)}
\\
&\leq C\left(\|df\|_{L^q(M;\RR^n)} + \|df_\infty\|_{L^q(M;\RR^n)} + 1\right)\|f - f_\infty\|_{W^{1,2}(M;\RR^n)}.
\end{align*}
Therefore,
\begin{align*}
\|f\|_{W^{1,q}(M;\RR^n)}
&\leq
\|f-f_\infty\|_{W^{1,q}(M;\RR^n)} + \|f_\infty\|_{W^{1,q}(M;\RR^n)}
\\
&\leq C\|df\|_{L^q(M;\RR^n)}\|f - f_\infty\|_{W^{1,2}(M;\RR^n)}
\\
&\quad + C\left(\|df_\infty\|_{L^q(M;\RR^n)}+1\right)\|f - f_\infty\|_{W^{1,2}(M;\RR^n)} + \|f_\infty\|_{W^{1,q}(M;\RR^n)}.
\end{align*}
Choosing $\eps = \eps (g,h,q) \leq 1/(2C)$ in \eqref{eq:Relative_harmonic_map_energy_bound} and applying rearrangement in the preceding inequality yields
\[
\|f\|_{W^{1,q}(M;\RR^n)}
\leq \|df_\infty\|_{L^q(M;\RR^n)} + 1 + 2\|f_\infty\|_{W^{1,q}(M;\RR^n)},
\]
as desired.
\end{proof}

It remains to complete the

\begin{proof}[Proof of Theorem \ref{mainthm:Relative_energy_gap_harmonic_maps} using Theorem \ref{thm:Lojasiewicz-Simon_gradient_inequality_energy_functional_Sobolev}]
Combining the inequalities \eqref{eq:Apriori_W2p_estimate_difference_two_harmonic_maps} and \eqref{eq:Apriori_W1q_estimate_harmonic_map} yields
\[
\|f-f_\infty\|_{W^{2,p}(M;\RR^n)} \leq C\left(1+\|f_\infty\|_{W^{1,q}(M;\RR^n)}\right)\|f - f_\infty\|_{W^{1,2}(M;\RR^n)}.
\]
We now fix $p \in (1,2)$ and $q=2p/(2-p)$ (say with $p=3/2$) and choose $\eps = \eps(f_\infty,g,h) \in (0,1]$ in \eqref{eq:Relative_harmonic_map_energy_bound} small enough that
\[
\eps C\left(1+\|f_\infty\|_{W^{1,q}(M;\RR^n)}\right) \leq \sigma,
\]
where the constant $\sigma \in (0,1]$ is as in Theorem \ref{thm:Lojasiewicz-Simon_gradient_inequality_energy_functional_Sobolev}. Consequently,
\[
\|f-f_\infty\|_{W^{2,p}(M;\RR^n)} < \sigma,
\]
and the hypothesis \eqref{eq:Lojasiewicz-Simon_gradient_inequality_harmonic_map_neighborhood_Sobolev} is satisfied. The {\L}ojasiewicz--Simon gradient inequality \eqref{eq:Lojasiewicz-Simon_gradient_inequality_harmonic_map_energy_functional_Sobolev} (with $d=k=2$) in Theorem \ref{thm:Lojasiewicz-Simon_gradient_inequality_energy_functional_Sobolev} therefore yields
\[
\|\sM(f)\|_{L^p(M;f^*TN)}
\geq
Z|\sE(f) - \sE(f_\infty)|^\theta.
\]
But $\sM(f)=0$ since $f$ is harmonic and thus $\sE(f) = \sE(f_\infty)$.
\end{proof}

It is possible to give an alternative proof of Theorem \ref{mainthm:Relative_energy_gap_harmonic_maps} using Theorem \ref{thm:Lojasiewicz-Simon_gradient_inequality_energy_functional_Holder} with the aid of an elliptic bootstrapping argument to fulfill the stronger hypothesis \eqref{eq:Lojasiewicz-Simon_gradient_inequality_harmonic_map_neighborhood_Holder}. We first observe that Lemma \ref{lem:Apriori_W2p_estimate_difference_two_harmonic_maps} can be strengthened to give

\begin{lem}[\Apriori $W^{k,p}$ estimate for the difference of two harmonic maps]
\label{lem:Apriori_Wkp_estimate_difference_two_harmonic_maps}
Let $(M,g)$ be a closed Riemann surface, $(N,h)$ a closed, smooth Riemannian manifold, $p \in (1, \infty)$ a constant, $k\geq 2$ an integer, and $f_\infty \in C^\infty(M; N)$ a harmonic map. Then there are constants $\eps=\eps(f_\infty,g,h,k,p) \in (0,1]$ and $C=C(f_\infty,g,h,k,p) \in [1,\infty)$ with the following significance. If $f \in C^\infty(M; N)$ is a harmonic map that obeys \eqref{eq:Relative_harmonic_map_energy_bound}, then
\begin{equation}
\label{eq:Apriori_Wkp_estimate_difference_two_harmonic_maps}
\|f-f_\infty\|_{W^{k,p}(M;\RR^n)} \leq C\|f - f_\infty\|_{W^{1,2}(M;\RR^n)}.
\end{equation}
\end{lem}

\begin{proof}
For $k=2$ and $p \in (1,2]$, the conclusion follows by combining \eqref{eq:Apriori_W2p_estimate_difference_two_harmonic_maps} and \eqref{eq:Apriori_W1q_estimate_harmonic_map}. For $k\geq 3$ and $p \in (1,\infty)$, the conclusion follows by taking derivatives of \eqref{eq:Harmonic_map_difference_equation} and applying a standard elliptic bootstrapping argument.
\end{proof}

We can now give the

\begin{proof}[Proof of Theorem \ref{mainthm:Relative_energy_gap_harmonic_maps} using Theorem \ref{thm:Lojasiewicz-Simon_gradient_inequality_energy_functional_Holder}]
For $p\in (1,\infty)$ and $\lambda \in (0,1)$ and large enough $k = k(g,p,\lambda) \geq 2$, there is a continuous Sobolev embedding, $W^{k,p}(M;\RR^n) \subset C^{2,\lambda}(M;\RR^n)$, and thus a constant $C = C(g,k,p,\lambda) \in [1,\infty)$ such that
\[
\|f-f_\infty\|_{C^{2,\lambda}(M;\RR^n)} \leq C\|f - f_\infty\|_{W^{k,p}(M;\RR^n)}.
\]
Combining the preceding inequality with \eqref{eq:Apriori_Wkp_estimate_difference_two_harmonic_maps} yields the bound
\[
\|f-f_\infty\|_{C^{2,\lambda}(M;\RR^n)} \leq C\|f - f_\infty\|_{W^{1,2}(M;\RR^n)},
\]
for a constant $C = C(f_\infty,g,h,k,p,\lambda) \in [1,\infty)$ .

We now fix $k,p,\lambda$ and choose $\eps = \eps(f_\infty,g,h) \in (0,1]$ in \eqref{eq:Relative_harmonic_map_energy_bound} small enough that $C\eps \leq \sigma$, where the constant $\sigma \in (0,1]$ is as in Theorem \ref{thm:Lojasiewicz-Simon_gradient_inequality_energy_functional_Holder}. Consequently,
\[
\|f-f_\infty\|_{C^{2,\lambda}(M;\RR^n)} < \sigma,
\]
and the hypothesis \eqref{eq:Lojasiewicz-Simon_gradient_inequality_harmonic_map_neighborhood_Holder} is satisfied. The {\L}ojasiewicz--Simon gradient inequality \eqref{eq:Lojasiewicz-Simon_gradient_inequality_harmonic_map_energy_functional_Holder} in Theorem \ref{thm:Lojasiewicz-Simon_gradient_inequality_energy_functional_Holder} therefore yields
\[
\|\sM(f)\|_{L^2(M;f^*TN)}
\geq
Z|\sE(f) - \sE(f_\infty)|^\theta.
\]
Again, $\sM(f)=0$ since $f$ is harmonic and thus $\sE(f) = \sE(f_\infty)$.
\end{proof}

%
%

\bibliography{master,mfpde}
\bibliographystyle{amsplain}

\end{document}